\documentclass[11pt]{article}
\usepackage{graphicx}
\usepackage{pgfplots}

\usepackage[colorlinks=true,citecolor=blue]{hyperref}
\usepackage{natbib,extarrows}
\usepackage{graphicx}
\usepackage{amsfonts}
\usepackage{amsmath}
\usepackage{amssymb}
\usepackage{url}
\usepackage{fancyhdr}
\usepackage{indentfirst}
\usepackage{enumerate}
\usepackage{titlesec}
\usepackage{amsthm}
\usepackage{subcaption}
\usepackage{dsfont}
\usepackage[misc]{ifsym}
\usepackage[T1]{fontenc}
\usepackage{enumitem}
\usepackage{todonotes}

\theoremstyle{definition}
\newtheorem{example}{Example}

\newtheorem{definition}{Definition}

\theoremstyle{plain}
\newtheorem{theorem}{Theorem}
\newtheorem{lemma}{Lemma}
\newtheorem{proposition}{Proposition}
\newtheorem{corollary}{Corollary}
\theoremstyle{remark}
\newtheorem{remark}{Remark}
    \allowdisplaybreaks

\usepackage{cases}

\theoremstyle{definition}

\def\N{\mathbb{N}}
\def\P{\mathbb{P}}
\def\p{\mathbb{P}}
\def\E{\mathbb{E}}

\def\R{\mathbb{R}}

\def\H{\mathcal{H}}

\DeclareMathOperator*{\esssup}{ess\text{-}sup}
\DeclareMathOperator*{\essinf}{ess\text{-}inf}

\newcommand{\VaR}{\mathrm{VaR}}

\newcommand{\bd}{\bar{\theta}}
\renewcommand{\(}{\left(}
\renewcommand{\)}{\right)}

\usepackage[onehalfspacing]{setspace}

\usepackage{bm}
\usepackage{tikz-qtree}
\usepackage{tikz}

\def\id{\mathds{1}}

 \pgfdeclarepatternformonly{south east lines}{\pgfqpoint{-0pt}{-0pt}}{\pgfqpoint{3pt}{3pt}}{\pgfqpoint{3pt}{3pt}}{
    \pgfsetlinewidth{0.4pt}
    \pgfpathmoveto{\pgfqpoint{0pt}{3pt}}
    \pgfpathlineto{\pgfqpoint{3pt}{0pt}}
    \pgfpathmoveto{\pgfqpoint{.2pt}{-.2pt}}
    \pgfpathlineto{\pgfqpoint{-.2pt}{.2pt}}
    \pgfpathmoveto{\pgfqpoint{3.2pt}{2.8pt}}
    \pgfpathlineto{\pgfqpoint{2.8pt}{3.2pt}}
    \pgfusepath{stroke}}

\begin{document}

\title{Risk aggregation and stochastic dominance for a class of heavy-tailed distributions}

\author{Yuyu Chen\thanks{Department of Economics, University of Melbourne,  Australia. \Letter~{\scriptsize\url{yuyu.chen@unimelb.edu.au}}}
	\and Seva Shneer \thanks{Department of Actuarial Mathematics and Statistics, Heriot-Watt University,  UK. \Letter~{\scriptsize\url{V.Shneer@hw.ac.uk}}}
}

\maketitle

\begin{abstract}
We introduce a new class of heavy-tailed distributions for which any weighted average of independent and identically distributed random variables is larger than one such random variable in (usual) stochastic order. We show that many commonly used extremely heavy-tailed (i.e., infinite-mean) distributions, such as the Pareto, Fr\'echet, and Burr distributions, belong to this class. 
The established stochastic dominance relation can be further generalized to allow negatively dependent or non-identically distributed random variables. In particular, the weighted average of non-identically distributed random variables dominates their distribution mixtures in stochastic order.

\end{abstract}

\section{Introduction}
Distributions with infinite mean are ubiquitous in the realm of banking and insurance, and they are particularly useful in modeling catastrophic losses (\cite{IJW09}), operational losses (\cite{moscadelli2004modelling}), costs of cyber risk events (\cite{EW19}),  and financial returns from technology innovations (\cite{silverberg2007size}); see also \cite{CW25} for a list of empirical examples of distributions with infinite mean. 

As the world is arguably finite (e.g., any loss is bounded by the total wealth in the world), why should we use models with infinite mean as mathematical tools? The main reason is that infinite-mean models often fit extremely heavy-tailed datasets better than finite-mean models. Moreover, the sample mean of iid samples of heavy-tailed data may not converge or may even tend to infinity as the sample size increases. Therefore, it is not sufficient to conclude that infinite-mean models are unrealistic by the finiteness of the sample mean. Indeed, models with infinite moments are not ``improper" as emphasized by \cite{mandelbrot2013fractals}, and they have been extensively used in the financial and economic literature (see \cite{mandelbrot2013fractals} and \cite{C01}). 

This paper focuses on establishing some stochastic dominance relations for infinite-mean models.
For two random variables $X$ and $Y$, $X$ is  said to be smaller than $Y$ in stochastic order, denoted by $X \le_{\rm st } Y$, if
$\P(X \le x) \ge \P(Y\le x)$ for all  $x \in \R$; see  \cite{MS02} and \cite{SS07} for extensive accounts of properties of stochastic dominance. Let $X$ be a positive one-sided stable random variable with infinite mean and
$X_{1},\dots,X_{n}$ be iid copies of $X$. 
For a nonnegative vector $\(\theta_{1},\dots,\theta_n\)$ with $\sum_{i=1}^n\theta_i=1$, \cite{ibragimov2005new} showed that  
\begin{equation}
 \label{eq:*}
       X\le_{\rm st}\theta_{1}X_{1}+\dots+\theta_{n}X_{n}.
\end{equation} 
Recently,  \cite{arab2024convex}, \cite{CEW24a}, and \cite{muller2024some} have shown that  inequality \eqref{eq:*} holds for more general classes of distributions. The case of two Pareto random variables with tail parameter 1/2 was studied in Example 7 of \cite{embrechts2002correlation}; see Section \ref{sec:SHdist} for the precise definition of the Pareto distribution. 

Inequality \eqref{eq:*} provides very strong implications in decision making as it surprisingly holds in the strongest form of risk comparison. If $X_1,\dots,X_n$ are treated as losses in a portfolio selection problem, any agent who prefers less loss will choose to take one of $X_1,\dots,X_n$ instead of allocating their risk exposure over different losses. This observation is counterintuitive, contrasting with the common belief that diversification reduces risk. Other applications of \eqref{eq:*} include optimal bundling problems (\cite{IW10}) and risk sharing (\cite{CEW24b}). 
 
In this paper, we will study \eqref{eq:*} where $X_1,\dots,X_n$ are possibly negatively dependent, a case not considered in \cite{ibragimov2005new},  \cite{arab2024convex}, and \cite{muller2024some}. \cite{CEW24a} have shown that \eqref{eq:*} also holds for weakly negatively associated super-Pareto random variables $X_{1},\dots,X_{n}$.  The class of super-Pareto random variables is quite broad and can be obtained by applying increasing and convex transforms to a Pareto random variable with tail parameter $1$. Examples of super-Pareto distributions include the Pareto, generalized Pareto, Burr, paralogistic, and log-logistic distributions, all with infinite mean. 

This paper aims to further generalize the result of \cite{CEW24a} in two aspects: the marginal distribution and the dependence structure of $(X_{1},\dots,X_{n})$. In Section  \ref{sec:SHdist}, we first introduce a new class of distributions, which  has several nice properties (Propositions \ref{prop:superHT} and \ref{prop:mixture}) and includes the class of super-Pareto distributions as a special case. 
Within this class of distributions,
we show in Theorem \ref{thm:main} that \eqref{eq:*} holds for identically distributed  random variables $X_1,\dots,X_n$ that are negatively lower orthant dependent (\cite{block1982some}).  It is well known that the behavior of the sum of extremely heavy-tailed random variables is dominated by the maximum of the summands (see \cite{EKM97}). Therefore, a possible reason why \eqref{eq:*} is preserved when transitioning from independence to negative dependence is because under negative dependence,  random variables that take small to moderate values will push the other random variables to take large values with a larger probability, 
 leading to a stochastically larger $\sum_{i=1}^n\theta_iX_i$. The situation is different for positively dependent random variables; see Remark  \ref{re:positive}.
As negative lower orthant dependence is more general than weak negative association,  Theorem 1 (i) of \cite{CEW24a} is implied by Theorem \ref{thm:main}. Remarkably, while Theorem \ref{thm:main} is more general, it is shown by a much more concise proof.

In Section \ref{sec:extensions}, we proceed to study \eqref{eq:*} given non-identically distributed random variables $X_1,\dots,X_n$. Since $X_1,\dots,X_n$ do not follow the same distribution, the choice of $X$ becomes unclear. A possible choice is to let $X$ follow the generalized mean of the distributions of $X_1,\dots,X_n$. 
 A special case is the arithmetic mean, which leads to the commonly used distribution mixture models.  Considering a rather large class of distributions,
Theorem \ref{thm:frechet} shows that \eqref{eq:*} holds if the distribution of $X$ is the generalized mean with non-negative power of the distributions of $X_1,\dots,X_n$. To our best knowledge, Theorem \ref{thm:frechet} is the first attempt to establish a non-trivial version of  \eqref{eq:*} for non-identically distributed random variables.

The rest of the paper is organized as follows. In Section \ref{sec:SD}, we present some first observations on \eqref{eq:*}.  Sections \ref{sec:SHdist} and  \ref{sec:extensions} present the main results. Section \ref{sec:comparison} compares our results with  the literature. Section \ref{sec:con} concludes the paper. The appendix contains the proofs of Propositions \ref{prop:superHT} and \ref{prop:mixture} as well as some examples in the new class of distributions.

\subsection{Notation, conventions and definitions}

In this section, we collect some notation and conventions used throughout the rest of the paper and remind the reader of some well-known definitions.

A function $f$ on $(0,\infty)$ is said to be \emph{subadditive} if $f(x+y)\le f(x)+f(y)$ for any $x,y>0$. If the inequality is strict, we say $f$ is  \emph{strictly subadditive}. For a random variable $X\sim F$, denote by $\essinf X$ ($\essinf F$) and $\esssup X$ ($\esssup F$) its essential infimum and essential supremum. Denote by $\Delta_n$   the standard simplex, that is, $\Delta_n=\{\bd \in [0,1]^n:  \sum_{i=1}^n \theta_i=1\}$, where we use notation $\bd$ for a vector $(\theta_1,\dots,\theta_n)$. Let $\Delta_n^+=\Delta_n\cap (0,1)^n$. We will also use $[n]$ to denote the set of indices $1,\dots,n$. For a distribution function $F$, its generalized inverse is defined as
$$F^{-1}(p)=\inf\{t\in\R:F(t)\geq p\},~ p\in(0,1).$$

\begin{definition}
We say that a random variable $X$ is smaller than a random variable $Y$ in \emph{stochastic order}, denoted by $X \le_{\rm st } Y$, if
$
\P(X \le x) \ge \P(Y\le x)$ for all  $x \in \R$.
For random variables $X$ and $Y$ with support $[c,\infty)$ where $c\in\R$, we write $X<_{\rm st}Y$ if $\p(X\le x)>\p(Y\le x)$ for all $x>c$.

\end{definition}

\section{Some observations on the stochastic dominance}\label{sec:SD}

Throughout the paper, we work with random variables which are almost surely non-negative.

The main focus of the paper is on studying random variables $X$ such that
\begin{equation} \label{eq:main_property}
X \le_{\rm st} \theta_1 X_1 + \dots + \theta_n X_n\mbox{~for all $\bd \in \Delta_n$,}\tag{SD}
\end{equation}
 where $X_1, \dots, X_n$ are  independent or negatively dependent with the  marginal laws  equal to $X$ (see Section \ref{sec:NLOD} for the precise definition of negative dependence). We will also say that a distribution $F$ satisfies property \eqref{eq:main_property} if a random variable $X\sim F$ satisfies it. If some of $\theta_1,\dots,\theta_n$ are 0, we can simply reduce the dimension of our problem. Therefore, for most of our results, we will assume $\bd\in\Delta_n^+$.

Since \eqref{eq:main_property} holds if a constant is added to $X$, we will, without loss of generality, only consider random variables with essential infimum $0$.
We will also be interested in distributions, and random variables, for which property \eqref{eq:main_property} holds with a strict inequality.
Let us start by formulating and providing some straightforward observations of \eqref{eq:main_property}. 

\begin{proposition} \label{prop:first_properties}
Assume that random variables $X$ and $Y$ satisfy property \eqref{eq:main_property} and are independent. Then the following statements hold.
\begin{enumerate}[label=(\roman*)]

\item $\E(X) = \infty$ or $X$ is a constant.

\item A random variable $aX + b$ with $a \ge 0$ and  $b \in \R$ satisfies \eqref{eq:main_property}.

\item Random variables $\max\{X,c\}$ and $\max\{X,Y\}$  satisfy \eqref{eq:main_property}, with $c \ge 0$.

\item A random variable $g(X)$ with a convex non-decreasing function $g$ satisfies \eqref{eq:main_property}. In addition, if $X$ satisfies \eqref{eq:main_property} with a strict inequality, $g$ is convex and strictly increasing, then $g(X)$ also satisfies \eqref{eq:main_property} with a strict inequality.

\end{enumerate}

\end{proposition}

\begin{proof}
\begin{enumerate}[label=(\roman*)]
\item This is implied by Proposition 2 of \cite{CEW24a}.
\item
The proof is straightforward and is omitted.
\item
 We will prove only the stronger property for the maximum of two random variables. Let $X_{1},\dots,X_{n}$ follow the distribution of $X$, $Y_{1},\dots,Y_{n}$ follow the distribution of $Y$, and $\{X_i\}_{i\in[n]}$ and $\{Y_i\}_{i\in[n]}$ be independent. For $x\in\R$ and $\bar \theta\in\Delta_n$, we have
\begin{align*}
\P(\max\{X,Y\} \le x) & = \P(X\le x) \P(Y\le x) \ge \P\left(\sum_{i=1}^n \theta_i X_i \le x\right) \P\left(\sum_{i=1}^n \theta_i Y_i \le x\right)\\
& = \P\left(\sum_{i=1}^n \theta_i X_i \le x, \sum_{i=1}^n \theta_i Y_i \le x\right) = \P\left(\max\left\{\sum_{i=1}^n \theta_i X_i, \sum_{i=1}^n \theta_i Y_i\right\} \le x\right) \\
& \ge \P\left(\sum_{i=1}^n \theta_i \max\{X_i,Y_i\} \le x\right).
\end{align*}


\item
Since $g$ is convex and non-decreasing, $g(X)\le_{\rm st}g(\sum_{i=1}^n\theta_{i}X_{i})\le \sum_{i=1}^n\theta_{i}g(X_{i})$,  where the first inequality holds as stochastic order is preserved under non-decreasing transforms and the second inequality is to be understood in the almost sure (and therefore also stochastic) sense and is due to convexity of $g$.
\qedhere
\end{enumerate}
\end{proof}

Properties (ii)-(iv) above demonstrate that, even if one knows only several random variables satisfying \eqref{eq:main_property}, it is possible to construct many more. Of special interest is property (iii), which does not require any specific distributional properties of $X$ and $Y$ apart from property \eqref{eq:main_property}.


\section{A class of heavy-tailed distributions and stochastic dominance}\label{sec:SHdist}

In this section, we introduce a new class of heavy-tailed distributions.   
We  explore several properties of this class and demonstrate that it contains many well-known distributions with infinite mean. 
We then prove that all distributions in this class satisfy property \eqref{eq:main_property}.
Along with the results of Proposition \ref{prop:first_properties}, this shows that the class of distributions satisfying property \eqref{eq:main_property} is large.

\subsection{A class of heavy-tailed distributions}
As has already been noted, we can, without loss of generality, consider random variables whose essential infimum is zero. For a random variable $X\sim F$ with $\essinf X=0$, we have  $F(x)>0$ for all $x>0$.

\begin{definition}
Let $F$ be a distribution function with $\essinf F=0$ and let
$h_F(x)=-\log F(1/x)$ for $x\in(0,\infty)$.
We say that $F$ belongs to $\mathcal H$, denoted by $F\in\mathcal H$, if $h_F$ is subadditive. We write $F\in\mathcal H_s$ if $h_F$ is strictly subadditive. For  $X\sim F$, we also write $X\sim \mathcal H$ (resp.~$X\sim \mathcal H_s$) if $F\in\mathcal H$ (resp.~$F\sim \mathcal H_s$).
\end{definition}

\begin{remark}
By properties of subadditive functions (e.g., Theorem 7.2.4 and Theorem 7.2.5 of \cite{hille1996functional}), $F\in \H$ if $h_F(x)/x$ is decreasing or $h_F$ is concave. 
\end{remark}

In the case of continuous distribution $F$, $F\in \mathcal H$ holds if and only if  the survival function of $1/X$ is log-superadditive where $X\sim F$.
We will see later that all distributions in $\H$ have infinite mean and because of that we say $\H$ is a class of heavy-tailed distributions. Note that the definition of heavy-tailed distributions varies in different contexts; see, e.g., Remark \ref{rem:HT}.
Below are some examples in class $\H$.
 
\begin{example}[Fr\'echet distribution]
For $\alpha>0$, the Fr\'echet distribution, denoted by Fr\'echet$(\alpha)$, is defined as
$$F(x)=\exp(-x^{-\alpha}), ~~x>0.$$
If  $\alpha\le 1$, $F$ has infinite mean. It is easy to check that $F\in\H$ if $\alpha\le 1$ and $F\in\H_s$ if $\alpha<1$, since for any $x,y>0$, 
$$\frac{h_F(x)+h_F(y)}{h_F(x+y)}=\left(\frac{x}{x+y}\right)^\alpha+\left(1-\frac{x}{x+y}\right)^\alpha\ge 1.$$
As $h_F$ is additive when $\alpha=1$, Fr\'echet$(1)$ distribution can be thought as a ``boundary'' of class $\H$.
\end{example}

\begin{example}[Pareto$(1)$ distribution]\label{ex:Pareto1}
For $\alpha>0$, the Pareto distribution, denoted by Pareto$(\alpha)$, is defined as
$$F(x)=1-\frac{1}{(x+1)^\alpha},~~ x>0.$$
Pareto$(\alpha)$ distributions have infinite mean if $\alpha\le 1$. Taking second derivative of $h_F$ when $\alpha=1$, we have 
$h''_F(x)=-1/(x+1)^2.$
Hence $h_F$ is concave and ${\rm Pareto}(1)\in\H_s$.
\end{example}

We can show that Pareto$(\alpha)$ with $\alpha\le 1$, as well as many other infinite-mean distributions in Table \ref{t1} also belong to $\H$ either directly using the definition, or using some closure properties of $\H$ in Propositions \ref{prop:superHT} and \ref{prop:mixture} provided below; see Appendix for detailed derivations of examples in Table \ref{t1} and the proofs of Propositions \ref{prop:superHT} and \ref{prop:mixture}.


   \begin{table}[h!]
  {
\begin{center}
\begin{tabular}{|c |c | c|}
\hline
 & Distribution functions & Parameters  \\ 
 \hline
Fr\'echet distribution & $F(x)=\exp(-x^{-\alpha}), ~~x>0$ & $\alpha\le 1$ \\  
\hline
Pareto distribution  & $F(x)=1-(x+1)^{-\alpha},~~ x>0$ & $\alpha\le 1$    \\
 \hline
 Generalized Pareto distribution  & $F(x)=1-\(1+\xi(x/\beta)\)^{-1/\xi},~~ x>0$ & $\xi\ge 1$    \\
 \hline
 Burr distribution  & $F(x)=1-\(x^\tau+1\)^{-\alpha}, ~~x>0$ & $\alpha,\tau\le 1$    \\
 \hline
Inverse Burr distribution   & $F(x)=\(x^\tau/(x^\tau+1)\)^\alpha,~~ x>0$ & $\alpha>0$, $\tau\le 1$    \\
 \hline
Log-Pareto distribution   & $F(x)=1-(\log(x+1)+1)^{-\alpha},~~x>0.$ & $\alpha\le 1$    \\
 \hline
Stoppa distribution   & $F(x)=\(1-(x+1)^{-\alpha}\)^\beta, ~~x>0$ & $\alpha\le 1$, $\beta>0$    \\
 \hline

\end{tabular}
 \caption{Examples of  distributions in $\H$}
 \label{t1}
\end{center}
}
\end{table}
\begin{proposition}\label{prop:superHT}
Let $X\sim F$ where $F\in \H$. The following statements hold.
\begin{enumerate}[label=(\roman*)]
\item 
If $F$ is strictly increasing on $[0,\infty)$ and $\p(X<\infty)=1$, then $F$ is continuous on $[0,\infty)$.
\item 
For $\beta>0$, $F^\beta\in \H$.

\item If, in addition, a random variable $Y\sim G$, where $G\in\H$, is independent of $X$, then $\max\{X,Y\}\in \H$. In terms of distribution functions, if $F,G\in\H$, then $FG\in\H$.

\item For a  non-decreasing, convex, and non-constant  function $f: \R_+\rightarrow \R_+$ with $f(0)=0$,  $f(X)\in\H$. 
\end{enumerate}
\end{proposition}


 \begin{proposition}\label{prop:mixture}
  Let $\bd \in \Delta_n^+$. If distribution functions $F_1,\dots,F_n\in\H$ and $F_1\le_{\rm st }\dots\le_{\rm st }F_n$, then
  $\sum_{i=1}^n\theta_iF_i\in\H$.  
  \end{proposition}
  
   It is clear that the various transforms of distributions in Proposition \ref{prop:superHT} from our class generate many different distributions, showing that the class $\mathcal H$ is indeed rather large.  
  Suppose that $F_1,\dots,F_n$ are Pareto distributions with possibly different tail parameters $0<\alpha_1,\dots,\alpha_n\le 1$. As $F_1,\dots,F_n$ are comparable in stochastic order, by Proposition \ref{prop:mixture}, mixtures of $F_1,\dots,F_n$ are in $\H$.


\subsection{Negative lower orthant dependence}\label{sec:NLOD}

The notion of negative dependence below will be used to establish the main result of this section.

\begin{definition}[\cite{block1982some}]
Random variables $X_1,\dots,X_n$ are \emph{negatively lower orthant dependent} (NLOD)  if for all $x_1,\dots,x_n \in\R$,  $\p(X_1\le x_1,\dots,X_n\le x_n)\le \prod_{i=1}^n\p(X_i\le x_i)$.
\end{definition}
Negative lower orthant dependence includes independence as a special case. It is  commonly used in various research areas
 and it is implied by other popular notions of negative dependence in the literature, such as negative association (\cite{AS81}  and \cite{JP83}), negative orthant dependence (\cite{block1982some}), and negative regression dependence (\cite{L66} and \cite{block1985concept}) see, e.g., \cite{chi2022multiple} for the implications of these notions. 
 \subsection{Main result}
 \begin{theorem}\label{thm:main}
If a random variable $X\in \H$ and random variables $X_1,\dots,X_n$ are NLOD with marginal laws equal to $X$, then for $\bd \in\Delta_n^+$,
\begin{equation} \label{eq:maineq1}
 X \le_{\rm st} \theta_1 X_1 + \dots + \theta_n X_n.
\end{equation} 
If $X\in\H_s$, then
 $X <_{\rm st}\sum_{i=1}^n\theta_{i}X_{i}$.
\end{theorem}
\begin{proof}
Let $X\sim F$ and $\bar \theta\in\Delta_n^+$. We have, for all $x>0$,
\begin{align*}
\p\(\sum_{i=1}^n\theta_{i}X_{i}\le x\)&\le \p(\theta_1X_1\le x,\dots,\theta_nX_n\le x) \le \prod_{i=1}^n F\(\frac{x}{\theta_i}\)= \prod_{i=1}^n \exp\(-h_F\(\frac{\theta_i}{x}\)\)
\\
&= \exp\(- \sum_{i=1}^nh_F\(\frac{\theta_i}{x}\)\)\le  \exp\(- h_F\(\sum_{i=1}^n\frac{\theta_i}{x}\)\)=\exp\(- h_F\(\frac{1}{x}\)\)=F(x).
\end{align*}
The strictness statement is straightforward. The proof is complete.
\end{proof}

An immediate consequence of Theorem \ref{thm:main} and Proposition \ref{prop:first_properties} (i) is that all distributions in $\H$ have infinite mean.

\begin{remark}[Value-at-Risk]
 One regulatory risk measure in insurance and finance is Value-at-Risk (VaR).
  For a random variable $X\sim F$ and $p\in (0,1)$, VaR is defined as $\VaR_{p}(X)=F^{-1}(p)$. For two random variables $X$ and $Y$, it is well known that $X\le_{\rm st} Y$ if and only if $\VaR_p(X)\le \VaR_p(Y)$ for all $p\in(0,1)$. Note that VaR is comonotonic-additive; a risk measure $\rho$ is \emph{comonotonic-additive} if $\rho(Y+Z)=\rho(Y)+\rho(Z)$ for comonotonic random variables $Y$ and $Z$.\footnote{Random variables $Y$ and $Z$ are comonotonic if there exists a random variable $U$ and
two increasing functions $f$ and $g$ such that $Y = f(U)$ and $Z= g(U)$ almost surely.} Then we have $\VaR_p(X)=\sum_{i=1}^n\VaR_p(\theta_iX_i)$ for identically distributed random variables $X_1,\dots,X_n$. By Theorem \ref{thm:main}, superadditivity
of VaR holds for $\bd\in\Delta_n^+$ and identically distributed risks $X_1,\dots,X_n\in \H$ that are NLOD: For all $p\in(0,1)$,
$ \VaR_p(\theta_1 X_{1})+\dots+\VaR_p(\theta_n X_{n})\le \VaR_p(\theta_{1}X_{1}+\dots+\theta_{n}X_{n}).$ More generally, the superadditivity property holds for any comonotonic-additive risk measure that is consistent with stochastic order.

\end{remark}

\begin{remark}[Heavy-tailed distributions]\label{rem:HT}
A distribution $F$ is said to be  \emph{heavy-tailed} in the sense of \cite{falk2010laws} with tail parameter $\alpha>0$, if
$\overline F(x)= L(x)x^{-\alpha}$
where $L$ is a \emph{slowly varying} function, that is, $L(tx)/L(x)\rightarrow 1$ as $x\rightarrow \infty$ for all $t>0$. For iid heavy-tailed random variables $X_1,X_2,\dots\sim F$, if there exist sequences of constants $\{a_n\}$ and $\{b_n\}$ where $b_n>0$ such that $(\max\{X_1,\dots,X_n\}-a_n)/b_n$ converges to the Fr\'echet distribution, $F$ is said to be in the maximum domain attraction of the Fr{\'e}chet distribution. It is known in the Extreme Value Theory (\cite{EKM97}) that a distribution is in the maximum domain of attraction of the Fr{\'e}chet distribution if and only if the distribution is heavy-tailed. Note that for a heavy-tailed random variable $X$ with $\alpha\le 1$, $\E(|X|)=\infty$. An interesting property of heavy-tailed risks with infinite mean is the asymptotic superadditivity of VaR: If $X_1,\dots,X_n$ are iid and heavy-tailed  with tail parameter $\alpha< 1$, 
$$\lim_{p\rightarrow 1}\frac{\VaR_p(X_1+\dots+X_n)}{\VaR_p(X_1)+\dots+\VaR_p(X_n)}> 1.$$
See, e.g., Example 3.1 of \cite{ELW09} for the claim above.
Heavy-tailed risks with infinite mean are not necessarily in $\H$ as the condition of $\H$ applies over the whole range of distributions whereas heavy-tailed distributions have power-law shapes only in their tail parts. On the other hand,  risks in $\H$ are not necessarily heavy-tailed in the sense of \cite{falk2010laws}. For instance, the survival distributions of log-Pareto risks are slowly varying functions. Distributions with slowly varying tails are called super heavy-tailed.  
\end{remark}

\begin{remark}[Convex order]
Besides stochastic order, another popular notion of stochastic dominance to compare risks is convex order. For two random variables $X$ and $Y$, $X$ is said to be smaller than $Y$ in \emph{convex order}, denoted by $X\le_{\rm cx}Y$, if $\E(u(X))\le \E(u(Y))$ for all convex functions $u$ provided that the expectations exist. The interpretation of $X\le_{\rm cx}Y$ is that $Y$ is more ``spread-out'' than $X$. If $X_1,\dots,X_n$ are iid and have a finite mean, by Theorem 3.A.35 of \cite{SS07}, for $\bd\in\Delta_n^+$, $\sum_{i=1}^n\theta_iX_i\le_{\rm cx}X_1$. Unlike Theorem \ref{thm:main}, this leads to a diversification benefit. Note that $\le_{\rm cx}$ is not suitable for the analysis of risks with infinite mean as the expectation of any increasing convex transform of these risks is infinity.

\end{remark}

\begin{remark}[Positive dependence]\label{re:positive}
One may expect positive dependence to make larger values of the sum in \eqref{eq:maineq1} more likely and thus the sum more likely to stochastically dominate a single random variable. We believe that this intuition does not hold due to the very heavy tails of the random variables under consideration. It is known, for instance, that very large values of the sum of iid random variables with heavy tails are likely caused by a single random variable taking a large value, while other random variables are moderate.  If random variables are positively dependent and some of them do not take large values, it makes others more likely to take moderate values too, hence positive dependence may hinder large values; see \cite{alink2004diversification} and \cite{mainik2010optimal} for such observations in some asymptotic senses. These phenomena can also be seen from the deadly risks considered by \cite{muller2024some}: For all $i\in[n]$, $\P(X_i=0)=1-p$ and $\p(X_i=\infty)=p$ where $p>0$. For $\overline \theta\in\Delta_n^+$, it is clear that $\sum_{i=1}^n\theta_iX_i=\infty$ as long as one of $X_1,\dots,X_n$ is $\infty$. 
If $X_1,\dots,X_n$ are \emph{positively lower orthant dependent} (PLOD), that is  $\p(X_1\le x_1,\dots,X_n\le x_n)\ge \prod_{i=1}^n\p(X_i\le x_i)$ for all $x_1,\dots,x_n \in\R$, then
\begin{align*}
\p\(\sum_{i=1}^n\theta_iX_i=\infty\)&=1-\p\(X_1=\dots=X_n=0\),\\
&=1-\p\(X_1\le 0,\dots,X_n\le0\)\le 1-\prod_{i=1}^n\p(X_i= 0).
\end{align*}
Hence, $\sum_{i=1}^n\theta_iX_i$ is stochastically smaller when  $X_1,\dots,X_n$ are PLOD compared to the case when $X_1,\dots,X_n$ are independent.
 The situation is reversed for NLOD random variables. However, \eqref{eq:main_property} still holds for PLOD risks $X_1,\dots,X_n$ as 
 $$\p\(\sum_{i=1}^n\theta_iX_i=\infty\)=1-\p\(X_1=\dots=X_n=0\)\ge 1-\p(X_1=0)=\p(X_1=\infty).$$ In \cite{CHWZ24},  \eqref{eq:main_property} is shown to  hold for infinite-mean Pareto random variables that are positively dependent via some specific Clayton copula.  
\end{remark}

\section{Weighted sums of non-identically distributed risks}\label{sec:extensions}

In the previous section, property \eqref{eq:main_property} is studied for risks with the same marginal distribution. We now look at the case when risks are not necessarily identically distributed.  Given non-identically distributed random variables $X_1, \dots, X_n$ and any $\bd\in\Delta_n$, the question is to study for which random variable $X$ the following property holds
\begin{equation} \label{eq:nonidentical}
X \le_{\rm st} \theta_1 X_1 + \dots + \theta_n X_n.
\end{equation}

To study this problem, we introduce the class of super-Fr\'echet distributions defined below.

 \begin{definition}
 A random variable $X$ is said to be  \emph{super-Fr\'echet} (or has a super-Fr\'echet distribution)
 if $X$ and $f(Y)$ have the same distribution, where $Y\sim$ {Fr\'echet}$(1)$ and $f$ is a strictly increasing and convex function with $f(0)=0$. 
 \end{definition}
 As convex transforms make the tail of random variables heavier, super-Fr\'echet distributions are more heavy-tailed than Fr\'echet$(1)$ distribution, and thus the name. It is easy to check that  a random variable $X$ with $\essinf X=0$ is  super-Fr\'echet if and only if
the function 
 $g: x\mapsto 1/(-\log \p(X\le x))
 $ is strictly increasing and concave 
 on $(0,\infty)$ with $\lim_{x\downarrow 0}g(x)=0$.

  As Fr\'echet$(1)$ distribution is in $\H$, by Proposition \ref{prop:first_properties} (iv), Super-Fr\'echet distributions are in $\H$. On the other hand, not all distributions in $\H$ are Super-Fr\'echet, which can be seen from the following example.
 \begin{example}\label{ex:step}
For $c>0$, define a distribution function on $(0,\infty]$:
$$F(x)=\exp(-c\lceil1/x\rceil),\mbox{~~for~~}x\in(0,\infty),$$
and $F(\infty)=1$. Then $h_F(x)=c \lceil x\rceil $, $x\in(0,\infty)$, is subadditive, and thus $F\in \H$. However, since $F$ is not a continuous distribution, it is not super-Fr\'echet. The distributions in this example are the so-called inverse-geometric distributions, also considered in Example 2.7 of \cite{arab2024convex}.
 \end{example}
 
 Fr\'echet distributions with infinite mean, as well as many other distributions in the following example, are super-Fr\'echet.

\begin{example}
Pareto,  Burr, paralogistic, and log-logistic random variables, all with infinite mean, are  super-Fr\'echet distributions. Since all these random variables can be obtained by
applying strictly increasing and convex transforms to Pareto$(1)$ random variables (see Appendix \ref{sec:SHex}), it suffices to show that a Pareto$(1)$ random variable is super-Fr\'echet. Write the Pareto$(1)$ distribution  as 
$F(x)=1-1/(x+1)=\exp(-1/g(x))$, $x>0$,
where $g(x)=1/\log(1+1/x)$. It is clear that $g$ is strictly increasing and $\lim_{x\downarrow 0}g(x)=0$. We show $g$ is concave 
 on $(0,\infty)$. We have 
 $$g''(x)=\frac{2 - (1 + 2 x) \log(1 + 1/x)}{x^2 (1 + x)^2 \log^3(1 + 1/x)}.$$
 Let 
 $r(x)=\log\(1/x+1\)-2/(1+2x)$, $x>0$. It is easy to verify that $r$ is strictly decreasing on $(0,\infty)$ and $r(x)$ goes to $0$ as $x$ goes to infinity. Thus $r(x)>0$ and $g''(x)<0$ for $x\in(0,\infty)$. 
\end{example}
 

We will assume $X_1,\dots,X_n$ in \eqref{eq:nonidentical} are super-Fr\'echet. Since $X_1,\dots,X_n$ may not have the same distribution, how to choose the distribution of $X$ is not clear. A perhaps natural candidate is the generalized mean of the distributions of $X_1,\dots,X_n$.  For $r\in \R\setminus \{0\}$, $n\in \N$, and $\mathbf w=(w_1,\dots,w_n)\in \Delta_n$,  the generalized $r$-mean  function is defined as
$$
M^{\mathbf w}_r(u_1,\dots,u_n) = \left(w_1 u_1^r+\dots+w_n u_n^r \right)^{1/r},~~~~~(u_1,\dots,u_n) \in (0,\infty)^{n}.
$$ 
The generalized $0$-mean function is the weighted geometric mean, that is,
$
M^{\mathbf w}_0(u_1,\dots,u_n) =\prod_{i=1}^n u_i^{w_i},
$
 which is also the limit of $M^{\mathbf w}_r$ as $r\to 0$. A generalized mean of distribution functions is a distribution function. In particular, if $r=1$, it leads to a distribution mixture model, i.e., if $X\sim M^{ \mathbf w}_1(F_1,\dots,F_n)$,  $X$ has the same distribution as $\sum_{i=1}^nX_i\id_{A_i}$ where $A_1,\dots,A_n$ are mutually exclusive, independent of $X_1,\dots,X_n$, and $\p(A_i)=w_i$ for all $i\in[n]$.

 \begin{theorem}\label{thm:frechet}
If
 $X_1,\dots,X_n$ are super-Fr\'echet and  NLOD with $X_i\sim F_i$, $i\in[n]$,  and $X\sim M^{\bar \theta}_r(F_1,\dots,F_n)$ for some $r\ge 0$, 
then for $\bd\in\Delta_n^+$, 
\begin{equation*} 
X \le_{\rm st} \theta_1 X_1 + \dots + \theta_n X_n.
\end{equation*}
\end{theorem}

\begin{proof}
Let $g_i(x)=1/(-\log F_i(x))$, $x>0$, for all $i\in[n]$.
As  
 $g_i$, $i\in[n]$, is strictly increasing and concave 
 on $(0,\infty)$ with $\lim_{x\downarrow 0}g_i(x)=0$, $ g_i(x)\ge \theta g_i(x/\theta)$ for all $x>0$ and $\theta\in(0,1)$. Then, for $\theta\in(0,1)$, 
 \begin{align}\label{eq:sfrechet}
 F_i\(\frac{x}{\theta}\)=\exp\(-g_i\(\frac{x}{\theta}\)^{-1}\)\le \exp\(-\theta g_i(x)^{-1}\)=F_i(x)^\theta.
 \end{align}
As $X_1,\dots,X_n$ are NLOD, by \eqref{eq:sfrechet}, for any  $x>0$, $(\theta_1,\dots,\theta_n)\in\Delta_n$, and $r\ge 0$, 
\begin{align*}
\p\(\sum_{i=1}^n\theta_{i}X_{i}\le x\)&\le \p(\theta_1X_1\le x,\dots,\theta_nX_n\le x)\le \prod_{i=1}^n F_i\(\frac{x}{\theta_i}\) \le\prod_{i=1}^n F_i\(x\)^{\theta_i}\\
&=M^{\bar \theta}_0(F_1(x),\dots,F_n(x))\le M^{\bar \theta}_r(F_1(x),\dots,F_n(x))=\p(X\le x). 
\end{align*}
The last inequality is because the generalized mean function is monotone in $r$; that is, given any  $\mathbf w\in \Delta_n$, $M^{\mathbf w}_r\le M^{\mathbf w}_s$ for $r\le s  $ (Theorem 16 of \cite{HLP34}). 
\end{proof}

\section{Comparison with existing results}\label{sec:comparison}
In this section, we compare our results with the literature. We first consider the case when $X_1,\dots,X_n$ in  \eqref{eq:main_property} are iid. In \cite{arab2024convex}, it is shown that \eqref{eq:main_property} holds for non-negative random variables that are InvSub; a random variable $X\sim F$ and its distribution is called  InvSub if $1-F(1/x)$ is subadditive. The class of InvSub distributions is larger than $\H$ as $h_F=-\log F(1/x)$ is subadditive implies that $1-F(1/x)$ is subadditive. \cite{muller2024some} showed that \eqref{eq:main_property} holds for super-Cauchy random variables; a random variable $X\sim F$ and its distribution is called super-Cauchy if $F^{-1}(G(x))$ is convex where $G$ is the standard Cauchy distribution function. Super-Cauchy distributions are continuous and can take positive values on the entire real line but they do not contain $\H$ as $\H$ includes non-continuous distributions (see Example \ref{ex:step}).  The proofs in both \cite{arab2024convex} and \cite{muller2024some} are short and elegant.

As our results cover the case of negatively dependent risks, for the rest of this section, we will focus on the comparison of our results with \cite{CEW24a}; to our best knowledge, \cite{CEW24a} is the only other paper that deals with negatively dependent risks.
 \cite{CEW24a} showed that  \begin{equation}\label{eq:CEW}
     \eqref{eq:main_property} \mbox{~holds for super-Pareto risks $X_1,\dots,X_n$ that are weakly negatively associated.}
 \end{equation} 
 For ease of comparison, definitions of  super-Pareto distributions and weak negative association are given in a slightly different form from \cite{CEW24a} below.
\begin{definition}
A random variable   $X$ and its distribution is \emph{super-Pareto} if  $X$ and $f(Y)$ have the same distribution for some non-decreasing, 
convex, and non-constant function $f$ with $f(0)=0$ and $Y\sim \mathrm{Pareto}(1)$.
\end{definition}
\begin{definition}
A set $S\subseteq \R^{k}$, $k\in \N$ is  \emph{decreasing} 
if $\mathbf x\in S$ implies $\mathbf y\in S$ for all $\mathbf y\le \mathbf x$. Random variables $X_1,\dots,X_n$ are \emph{weakly negatively associated} if  
 for any $i\in[n]$,  decreasing set $S  \subseteq  \R^{n-1}$, and $x\in \R$ with $\p(X_i\le x)>0$,  
 $$\p(\mathbf X_{-i} \in S,  X_i\le  x) \le \p(\mathbf X_{-i}\in S)\p( X_i\le  x).$$
 where $\mathbf X_{-i}=(X_1,\dots,X_{i-1}, X_{i+1},\dots,X_n)$.
\end{definition}
\begin{lemma}
If random variables $X_1,\dots,X_n$ are super-Pareto and weakly negatively associated, then $X_1,\dots,X_n\in\H$ and they are NLOD.
\end{lemma}
\begin{proof}
As Pareto$(1)$ risks are in $\H$ (see Example \ref{ex:Pareto1}), by Proposition \ref{prop:superHT} (iv), super-Pareto risks are in $\H$. Since $X_1,\dots,X_n$ are weakly negatively associated, 
for any $(x_1,\dots,x_n)\in \R^n$,
\begin{align*}
\p(X_1\le x_1,\dots,X_n\le x_n)\le \p(X_1\le x_1,\dots,X_{n-1}\le x_{n-1})\p(X_n\le n)\le \prod_{i=1}^n\p(X_i\le x_i).
\end{align*}
Thus, $X_1,\dots,X_n$ are NLOD.
\end{proof}
The above lemma shows that Theorem \ref{thm:main} implies \eqref{eq:CEW}, which is in Theorem 1 (i) of \cite{CEW24a}.  We present below a corollary, which leads to a similar result as Theorem 1 (ii) of \cite{CEW24a}.

\begin{corollary}\label{cor: random}
Suppose that a random variable $X\in\H$,  random variables $X_1,\dots,X_n$ are NLOD with marginal laws equal to $X$, and $\xi_1,\dots,\xi_n$ are any positive random variables independent of $X,X_1,\dots,X_n$ with $\sum_{i=1}^n\xi_i\le 1$. 
        If $\P(cX>t) \ge c\P(X>t) $ for all $c \in (0,1]$ and $t>0$, then for $x\ge 0$,
\begin{equation} \label{eq:random_weights_2}
\p\(\sum_{i=1}^n\xi_i X_i > x\) \ge \E\left(\sum_{i=1}^n \xi_i\right) \p(X>x).
\end{equation}
\end{corollary}

\begin{proof}
By Theorem \ref{thm:main} and the independence between $\xi_1,\dots,\xi_n$ and $X,X_1,\dots,X_n$, we have
\begin{align*}
\P\(\sum_{i=1}^n\xi_i X_i > x\) &= \E\left[\P\left(\sum_{i=1}^n\xi_i X_i > x|(\xi_1,\dots,\xi_n)\right)\right]\\
&\ge \E\left[\P\left(\left(\sum_{i=1}^n \xi_i\right) X > x|(\xi_1,\dots,\xi_n)\right)\right] \ge \E\left(\sum_{i=1}^n \xi_i\right) \P\left(X > x\right). \qedhere
\end{align*} 
\end{proof}

For $\bar \theta\in\Delta_n$, let $A_1,\dots,A_n$ be any events independent of $(X_1,\dots,X_n)$ and event $A$ be independent of $X$  satisfying $\p(A)=\sum_{i=1}^n \theta_i\p(A_i)$. If $X_1,\dots,X_n$ are financial losses, $A_1,\dots,A_n$ can be interpreted as the triggering events for these losses. Let 
$\xi_i = \theta_i \id_{A_i}$. By \eqref{eq:random_weights_2}, for $x\ge 0$, 
$$\p\(\sum_{i=1}^n\theta_i X_i\id_{A_i} > x\) \ge \E\left(\sum_{i=1}^n \theta_i \id_{A_i}\right) \p(X>x)=\p(A)\p(X>x)=\p(X\id_A>x),$$
which is equivalent to 
\begin{equation}\label{eq:events}
X\id_A\le_{\rm st}\sum_{i=1}^n\theta_i X_i\id_{A_i}.
\end{equation}
 Theorem 1 (ii) of \cite{CEW24a} showed \eqref{eq:events} with different assumptions from Corollary \ref{cor: random}; we refer readers to \cite{CEW24a} for more details.

\section{Conclusion}\label{sec:con}

In this paper, we provide some sufficient conditions for property \eqref{eq:main_property} to hold. One can see that the property, while very strong, holds for a remarkably large class of distributions. We have also shown that it remains valid for non-identically distributed random variables.

We conclude with some open questions. First, we are interested in understanding how close our sufficient conditions for \eqref{eq:main_property} are to the optimal ones, i.e., we would like to understand what conditions are necessary for  \eqref{eq:main_property}. 

Second, the definition of our class of heavy-tailed random variables seems to suggest that it is the distribution of $1/X$ that is of importance. We currently lack an intuitive explanation of this.

Finally, property \eqref{eq:main_property} raises the possibility that, for some random variables $X_1,\dots,X_n$ and two vectors $\bar{\eta}, \bar{\gamma} \in \R^n_+$, 
\begin{equation}\label{eq:Schur}
\eta_1 X_1 + \dots + \eta_n X_n\le_{\rm st }\gamma_1 X_1 + \dots + \gamma_n X_n,
\end{equation}
where $\bar \gamma$ is smaller than $\bar \eta$ in \emph{majorization order}; that is, $\sum^n_{i=1}\gamma_i =\sum^n_{i=1}\eta_i$ and $\sum^k_{i=1} \gamma_{(i)} \ge \sum^k_{i=1} \eta_{(i)}\,$ for $k\in [n-1]$
where $\gamma_{(i)}$ and $\eta_{(i)}$ represent the $i$th smallest order statistics of $\bar \gamma$ and $\bar \eta$. Clearly, \eqref{eq:Schur} implies \eqref{eq:main_property}.
It is well known that \eqref{eq:Schur} holds for iid stable random variables with infinite mean (see \cite{ibragimov2005new}) and it was recently shown to hold
 for iid Pareto random variables with infinite mean by \cite{CHWZ24}. It is of question whether \eqref{eq:Schur} can hold for a larger class of distributions.
 Note that the methods used in the current paper do not appear to be useful to address \eqref{eq:Schur} as we rely on the comparison of a sum with each of the summands. A more subtle approach to sums is required.

{
\appendix

\setcounter{table}{0}
\setcounter{figure}{0}
\setcounter{equation}{0}
\renewcommand{\thetable}{A.\arabic{table}}
\renewcommand{\thefigure}{A.\arabic{figure}}
\renewcommand{\theequation}{A.\arabic{equation}}

\setcounter{theorem}{0}
\setcounter{proposition}{0}
\renewcommand{\thetheorem}{A.\arabic{theorem}}
\renewcommand{\theproposition}{A.\arabic{proposition}}
\setcounter{lemma}{0}
\renewcommand{\thelemma}{A.\arabic{lemma}}

\setcounter{example}{0}
\renewcommand{\theexample}{A.\arabic{example}}

\setcounter{corollary}{0}
\renewcommand{\thecorollary}{A.\arabic{corollary}}

\setcounter{remark}{0}
\renewcommand{\theremark}{A.\arabic{remark}}
\setcounter{definition}{0}
\renewcommand{\thedefinition}{A.\arabic{definition}}
\newpage
\begin{center}
\Large   Appendices
\end{center}

\section{Proof of Proposition \ref{prop:superHT}}

\begin{enumerate}[label=(\roman*)]

  \item 
  As $h_F$ is subadditive and increasing, and $\lim_{x\downarrow 0}h_F(x)=0$, $h_F$ is continuous on $(0,\infty)$, and so is $F$ (see Remark 1 of \cite{matkowski1993subadditive}). The desired result is due to the right-continuity of $F$.
  \item 
  Proof of (ii) is straightforward and thus omitted.
  \item
  This is also straightforward.
\item
For $y\ge0$, let $f^{-1+}(y)=\inf\{x\ge 0: f(x)>y\}$ be the right-continuous generalized inverse of $f$ with the convention that $\inf \emptyset=\infty$. As $f$ is increasing, convex, and non-constant with $f(0)=0$, $f^{-1+}$ is strictly increasing and concave and $f^{-1+}(0)\ge 0$.   Therefore, by concavity of $f^{-1+}$ and $f^{-1+}(0)\ge 0$, it is clear that $f^{-1+}(tx)\ge tf^{-1+}(x)$ for any $x>0$ and $t\in(0,1]$.   For any  $a,b>0$, 
  \begin{align*}
f^{-1+}\(\frac{ab}{a+b}\)\(f^{-1+}\(a\)+f^{-1+}\(b\)\)&\ge \frac{a}{a+b}f^{-1+}\(b\)f^{-1+}\(a\)+\frac{b}{a+b}f^{-1+}\(a\)f^{-1+}\(b\)\\
&=f^{-1+}\(a\)f^{-1+}\(b\).
  \end{align*}
  Hence, we have
  \begin{align}\label{eq:ab}
  \(f^{-1+}\(\frac{ab}{a+b}\)\)^{-1}\le  \(f^{-1+}\(a\)\)^{-1}+\(f^{-1+}\(b\)\)^{-1}.
  \end{align}
 Denote by $F$ and  $G$ the distribution functions of $X$ and $f(X)$, respectively. Then 
  $G(x)=\p(f(X)\le x)=\p(X\le f^{-1+}(x))=F(f^{-1+}(x))$ for $x\ge0$.  By letting $g(x)=1/f^{-1+}(1/x)$ for $x>0$, we write $h_G=h_F\circ g$. 
  By  inequality \eqref{eq:ab}, for any $x,y>0$,
    \begin{align*}
    g(x+y)=\(f^{-1+}\(\frac{1/xy}{1/x+1/y}\)\)^{-1}\le \(f^{-1+}\(\frac{1}{x}\)\)^{-1}+\(f^{-1+}\(\frac{1}{y}\)\)^{-1}=g(x)+g(y).
    \end{align*}
    Therefore, $g$ is subadditive. As $h_F$ is subadditive and non-decreasing,  it is clear that $h_G=h_F\circ g$ is subadditive and we have the desired result. \qedhere
    \end{enumerate}

    \section{Proof of Proposition \ref{prop:mixture}}
    
  Let $G=\sum_{i=1}^n\theta_iF_i$. It suffices to show 
  \begin{align}\label{eq:mixture}
G\(\frac{xy}{x+y}\)\ge G(x)G(y) ~~~\mbox{for all $x,y>0$.}
  \end{align}
   For $n=2$, as $F_1$ and $F_2$ are super heavy-tailed, 
  \begin{align*}
G\(\frac{xy}{x+y}\)-G\(x\)G\(y\)&=\theta_1F_1\(\frac{xy}{x+y}\)+\theta_2F_2\(\frac{xy}{x+y}\)-G\(x\)G\(y\)\\
&\ge \theta_1F_1\(x\)F_1\(y\)+\theta_2F_2\(x\)F_2\(y\)-G\(x\)G\(y\)\\
&=\theta_1F_1\(x\)F_1\(y\)+\theta_2F_2\(x\)F_2\(y\)\\
&~~~-(\theta_1F_1(x)+\theta_2F_2(x))(\theta_1F_1(y)+\theta_2F_2(y))\\
&=\theta_1\theta_2(F_1\(x\)-F_2(x))(F_1\(y\)-F_2(y))\ge 0.
\end{align*}
Hence, \eqref{eq:mixture} holds for $n=2$. Next, assume that \eqref{eq:mixture} holds for $n=k-1$ where $k>3$ is an integer. Let $a=\sum_{i=1}^{k-1}\theta_iF_i(x)$,  $b=\sum_{i=1}^{k-1}\theta_iF_i(y)$, $c=a/(F_n(x)(1-\theta_n))$, and $d=b/(F_n(y)(1-\theta_n))$. For $n=k$, 
  \begin{align*}
G\(\frac{xy}{x+y}\)-G\(x\)G\(y\)&=\sum_{i=1}^k\theta_iF_i\(\frac{xy}{x+y}\)-G\(x\)G\(y\)\\
&=\sum_{i=1}^{k-1}\theta_iF_i\(\frac{xy}{x+y}\)+\theta_nF_n\(\frac{xy}{x+y}\)-G\(x\)G\(y\)\\
&=(1-\theta_n)\sum_{i=1}^{k-1}\frac{\theta_i}{1-\theta_n}F_i\(\frac{xy}{x+y}\)+\theta_nF_n\(\frac{xy}{x+y}\)-G\(x\)G\(y\)\\
&\ge(1-\theta_n)\(\sum_{i=1}^{k-1}\frac{\theta_i}{1-\theta_n}F_i\(x\)\)\(\sum_{i=1}^{k-1}\frac{\theta_i}{1-\theta_n}F_i\(y\)\)\\
&~~~+\theta_nF_n\(\frac{xy}{x+y}\)-G\(x\)G\(y\)\\
&\ge \frac{ab}{1-\theta_n}+\theta_nF_n(x)F_n(y)-(a+\theta_nF_n(x))(b+\theta_nF_n(y))\\
&=\frac{ab\theta_n}{1-\theta_n}+(\theta_n-\theta_n^2)F_n(x)F_n(y)-a\theta_nF_n(y)-b\theta_nF_n(x)\\
&=\theta_n(1-\theta_n)F_n(x)F_n(y)\(cd+1-c-d\).
\end{align*}
As $F_1\le_{\rm st }\dots\le_{\rm st }F_k$, $F_k\le \sum_{i=1}^{k-1}\theta_i/(1-\theta_n)F_i$. Thus $c,d\ge 1$ and $cd+1-c-d\ge 0$. The proof is complete by induction.

\section{Examples of distributions in the class $\H$}\label{sec:SHex}

In this section, we demonstrate that many well-known infinite-mean distributions are in $\H$.


\begin{example}[Pareto distribution]\label{ex:Pareto}
For $\alpha>0$, the Pareto distribution, denoted by Pareto$(\alpha)$, is defined as
$$F(x)=1-\frac{1}{(x+1)^\alpha},~~ x>0.$$
  If $\alpha=1$,  for $x,y>0$, 
\begin{align*}
h_F(x+y)-h_F(x)-h_F(y)=\log(x+y+1)-\log(x+1)-\log(y+1)\le0
\end{align*}
Thus, Pareto$(1)\in\H_s$. Then we note that
any Pareto$(\alpha)$ random variable $X$ can be written as $X=f(Z)$ where $Z\sim {\rm Pareto}(1)$ and $f(x)=(x+1)^{1/\alpha}-1$ for $x\ge 0$.  By Proposition \ref{prop:superHT} (iv), as $f$ is increasing and convex for $\alpha\le 1$, Pareto$(\alpha)\in\H_s$ if $\alpha\le 1$.
\end{example}

\begin{example}[Generalized Pareto distribution]
       The generalized Pareto distribution with  parameters $\xi\in \R$ and  $\beta>0$ is defined as 
\begin{equation*}\label{eq:GPD}
     F(x)=
    \begin{cases}
     1-\(1+\xi\frac{x}{\beta}\)^{-1/\xi},&\mbox{~~if $\xi\neq 0$}, \\
    e^{-x/\beta}, & \mbox{~~if $\xi= 0$},
    \end{cases}
\end{equation*}
where $x\in[0,\infty)$ if $\xi\ge 0$ and $x\in[0,-\beta/\xi)$ if $\xi< 0$.
     By the Pickands-Balkema-de Haan Theorem \citep{BD74, P75},  the generalized Pareto distributions are the only possible non-degenerate limiting distributions of the excess of random variables beyond a high threshold. If $\xi\ge 1$, $F\in\H$. This is by Proposition \ref{prop:superHT} (iv); that is, the generalized Pareto random variables with $\xi\ge 1$ can be obtained from location-scale transforms of Pareto$(1/\xi)$ random variables.
\end{example}

\begin{example}[Burr distribution]
  For $\alpha,\tau>0$, the Burr distribution is defined as
\begin{equation}\label{eq:Burr}
F(x)=1-\(\frac{1}{x^\tau+1}\)^\alpha, ~~x>0.\end{equation}
Let $Y\sim\rm{Pareto}(\alpha)$. Then $Y^{1/\tau}$ follows a Burr distribution.
If $\alpha,\tau\le1$,  the  Burr distribution is super-Pareto and hence $F\in\H$.
Special cases of Burr distributions are the paralogistic ($\alpha=\tau$) and the log-logistic ($\alpha=1$) distributions; see  \cite{KK03} and \cite{KPW12}.
\end{example}

\begin{example}[Inverse Burr distribution]
Suppose that $Y$ follows the Burr distribution \eqref{eq:Burr}. Then  $X=1/Y$ follows the inverse Burr distribution
\begin{equation*}
F(x)=\(\frac{x^\tau}{x^\tau+1}\)^\alpha,~~ x>0,
\end{equation*}
where $\alpha,\tau>0$.  If $\tau\le 1$, it is easy to check that the second derivative of  $h_F$ is always negative, and thus $h_F$ is subadditive. Hence $F\in\H$ if $\tau\le 1$. Note that the property of $\H$ may not always be preserved under the inverse transformation. For instance, if $Z$ follows a Fr\'echet distribution without finite mean, then $1/Z$ follows a Weibull distribution whose mean is always finite. 
\end{example}

\begin{example}[Log-Pareto distribution]
If $Y\sim \mathrm{Pareto}(\alpha)$, $\alpha>0$, then $X=\exp(Y)-1$ has a log-Pareto distribution (see p.~39 in \cite{A15}), with distribution function
$$F(x)=1-\frac{1}{(\log(x+1)+1)^\alpha},~~x>0.$$ 
If $\alpha\in(0,1]$, by Proposition \ref{prop:superHT} (iv), $F\in\H$.
\end{example}

\begin{example}[Stoppa distribution]
 For $\alpha>0$ and $\beta>0$, a (location-shifted) Stoppa distribution can be defined as
\begin{equation*}
F(x)=\(1-\frac{1}{(x+1)^\alpha}\)^\beta, ~~x>0.
\end{equation*}
Since a Stoppa distribution is a power transform of a Pareto distribution,
 by Proposition \ref{prop:superHT} (ii), if $\alpha\le 1$, $F\in\H$.  Power transforms have also been used to generalize 
Burr distributions (see p.~211 of \cite{KK03}).
\end{example}

}

\subsection*{Competing interests} 
The authors declare none.

\subsection*{Acknowledgements}  
The authors would like to thank two anonymous referees for their helpful comments.
The authors 
also 
thank Qihe Tang and Ruodu Wang for the valuable discussions and comments on a preliminary version of this paper. Yuyu Chen is supported by the 2025 Early Career Researcher Grant from the University of Melbourne.



\begin{thebibliography}{10}



%

\bibitem[\protect\citeauthoryear{Alam and  Saxena}{1981}]{AS81}
Alam, K. and Saxena, K. M. L. (1981). Positive dependence in multivariate distributions. \emph{Communications in Statistics-Theory and Methods}, {10}(12):1183--1196.

\bibitem[Alink et~al., 2004]{alink2004diversification}
Alink, S., L{\"o}we, M. and W{\"u}thrich, M.~V. (2004).
\newblock Diversification of aggregate dependent risks.
\newblock {\em Insurance: Mathematics and Economics}, 35(1):77--95.

 \bibitem[Arab et~al., 2024]{arab2024convex}
Arab, I., Lando, T., and Oliveira, P.~E. (2024).
\newblock Convex combinations of random variables stochastically dominate the
  parent for a large class of heavy-tailed distributions.
\newblock {\em arXiv:2411.14926}.

\bibitem[Arnold, 2015]{A15}
Arnold, B.~C. (2015).
\newblock {\em Pareto Distributions}.
\newblock Second Edition. CRC Press


%

\bibitem[Balkema and de Haan, 1974]{BD74}
Balkema, A. and de Haan, L. (1974).
\newblock Residual life time at great age.
\newblock {\em Annals of Probability}, 2:792--804.

%
%
%


  \bibitem[Block et~al., 1982]{block1982some}
Block, H.~W., Savits, T.~H. and Shaked, M. (1982).
\newblock Some concepts of negative dependence.
\newblock {\em Annals of Probability}, 10(3):765--772.

\bibitem[Block et~al., 1985]{block1985concept}
Block, H.~W., Savits, T.~H. and Shaked, M. (1985).
\newblock A concept of negative dependence using stochastic ordering.
\newblock {\em Statistics \& Probability Letters}, 3(2):81--86.


%
%
%
\bibitem[\protect\citeauthoryear{Chen et al.}{2024a}]{CEW24a}
Chen, Y., Embrechts, P. and Wang, R. (2024a). An unexpected stochastic dominance: Pareto distributions, dependence, and diversification. \emph{Operations Research}, forthcoming.

\bibitem[\protect\citeauthoryear{Chen et al.}{2024b}]{CEW24b}
Chen, Y., Embrechts, P. and Wang, R. (2024b). Risk exchange under infinite-mean Pareto models. \emph{arXiv:2403.20171}.

\bibitem[Chen et~al., 2025]{CHWZ24}
Chen, Y., Hu, T., Wang, R. and Zou, Z. (2025).
\newblock Diversification for infinite-mean Pareto models without risk aversion.
\newblock {\em European Journal of Operational Research}, 323(1):341--350.


\bibitem[\protect\citeauthoryear{Chen and Wang}{2025}]{CW25}
Chen, Y.  and Wang, R. (2025).  Infinite-mean models in risk management: Discussions and recent advances. \emph{Risk Sciences}, 1:100003.

\bibitem[Chi et~al., 2024]{chi2022multiple}
Chi, Z., Ramdas, A. and Wang, R. (2024).
\newblock Multiple testing under negative dependence.
\newblock {\em Bernoulli}, 31(2):1230--1255.


\bibitem[\protect\citeauthoryear{Cont}{Cont}{2001}]{C01}
Cont, R. (2001). Empirical properties of asset returns: stylized facts and statistical issues. \emph{Quantitative Finance}, {1}:223--236.


%
%
%
%
%
%
%
%
%
%
\bibitem[Eling and Wirfs, 2019]{EW19}
Eling, M. and Wirfs, J. (2019).
\newblock What are the actual costs of cyber risk events?
\newblock {\em European Journal of Operational Research}, {272}(3):1109--1119.
%
%
\bibitem[\protect\citeauthoryear{Embrechts et al.}{Embrechts et al.}{1997}]{EKM97}
Embrechts, P., Kl\"uppelberg, C. and Mikosch, T. (1997). \emph{Modelling Extremal Events for Insurance and Finance}. Springer, Heidelberg.

%

\bibitem[\protect\citeauthoryear{Embrechts et al.}{2009}]{ELW09}
Embrechts, P., Lambrigger, D. and W{\"u}thrich, M. (2009).
 Multivariate extremes and the aggregation of dependent risks: examples and counter-examples.
{\em Extremes}, {12}(2):107--127.
%
\bibitem[Embrechts et~al., 2002]{embrechts2002correlation}
Embrechts, P., McNeil, A. and Straumann, D. (2002).
\newblock Correlation and dependence in risk management: properties and pitfalls.
In \newblock {\em Risk Management: Value at Risk and Beyond} (Eds: Dempster), pp.~176--223, Cambridge University Press.


\bibitem[Falk et~al., 2011]{falk2010laws}
Falk, M., H{\"u}sler, J. and Reiss, R.-D. (2011).
\newblock {\em Laws of Small Numbers: Extremes and Rare Events}.
\newblock Springer Birkh{\"a}user Basel.



%
%

\bibitem[\protect\citeauthoryear{Hardy et al.}{Hardy et al.}{1934}]{HLP34}
Hardy, G. H., Littlewood, J. E. and P\'olya, G  (1934). \emph{Inequalities}. Cambridge University Press.

\bibitem[Hille and Phillips, 1996]{hille1996functional}
Hille, E. and Phillips, R.~S. (1996).
\newblock {\em Functional Analysis and Semi-groups}, volume~31.
\newblock American Mathematical Society.
%
%
%
%
%
%

  \bibitem[Ibragimov, 2005]{ibragimov2005new}
		Ibragimov, R. (2005).
		\newblock {New majorization theory in economics and martingale convergence
			results in econometrics}.  
		\newblock Ph.D. dissertation, Yale University, New Haven, CT.
		
\bibitem[\protect\citeauthoryear{Ibragimov et al.}{2009}]{IJW09}
Ibragimov, R., Jaffee, D. and Walden, J. (2009). Non-diversification traps in markets for catastrophic risk. \emph{Review of Financial Studies}, {22}:959--993.



\bibitem[Ibragimov and Walden, 2010]{IW10}
Ibragimov, R. and Walden, J. (2010).
\newblock Optimal bundling strategies under heavy-tailed valuations.
\newblock {\em Management Science}, 56(11):1963--1976.

\bibitem[\protect\citeauthoryear{Joag-Dev and Proschan}{Joag-Dev and Proschan}{1983}]{JP83}
Joag-Dev, K. and Proschan, F. (1983). Negative association of random variables with applications. \emph{Annals of Statistics}, 11(1):286--295.

\bibitem[Kleiber and Kotz, 2003]{KK03}
Kleiber, C. and Kotz, S. (2003).
\newblock {\em Statistical Size Distributions in Economics and Actuarial
  Sciences}.
\newblock John Wiley \& Sons.

\bibitem[\protect\citeauthoryear{Klugman et al.}{Klugman et al.}{2012}]{KPW12}
Klugman, S. A., Panjer, H. H. and Willmot, G. E. (2012). \emph{Loss Models: From Data to Decisions.} 4th Edition. John Wiley \& Sons.

\bibitem[\protect\citeauthoryear{Lehmann}{Lehmann}{1966}]{L66}
Lehmann, E. L. (1966). Some concepts of dependence. \emph{Annals of Mathematical Statistics},  {37}(5):1137--1153.


%
%




%
%
\bibitem[Mainik and R{\"u}schendorf, 2010]{mainik2010optimal}
Mainik, G. and R{\"u}schendorf, L. (2010).
\newblock On optimal portfolio diversification with respect to extreme risks.
\newblock {\em Finance and Stochastics}, 14:593--623.


\bibitem[Mandelbrot, 1997]{mandelbrot2013fractals}
Mandelbrot, B.~B. (1997).
\newblock {\em Fractals and Scaling in Finance: Discontinuity, Concentration,
  Risk.}
\newblock Springer, New York.






\bibitem[Matkowski and {\'S}wi{\k{a}}tkowski, 1993]{matkowski1993subadditive}
Matkowski, J. and {\'S}wi{\k{a}}tkowski, T. (1993).
\newblock On subadditive functions.
\newblock {\em Proceedings of the American Mathematical Society},
  119(1):187--197.




\bibitem[Moscadelli, 2004]{moscadelli2004modelling}
Moscadelli, M. (2004).
\newblock The modelling of operational risk: Experience with the analysis of the data collected by the Basel committee. Technical Report 517. \emph{SSRN}: 557214.


\bibitem[M{\"u}ller, 2024]{muller2024some}
M{\"u}ller, A. (2024).
\newblock Some remarks on the effect of risk sharing and diversification for
  infinite mean risks.
  \newblock {\em  arXiv:2411.10139}.

\bibitem[M\"uller and Stoyan, 2002]{MS02} {M\"uller, A. and Stoyan, D.} (2002). \newblock \emph{Comparison Methods for Stochastic Models and Risks}. \newblock Wiley, England.
%
%
%
%
%
%


\bibitem[Pickands, 1975]{P75}
Pickands, J. (1975).
\newblock
Statistical inference using extreme order statistics.
\newblock {\em Annals of Statistics}, {3}:119--131.
%
%








\bibitem[\protect\citeauthoryear{Shaked and Shanthikumar}{Shaked and Shanthikumar}{2007}]{SS07}
Shaked, M. and Shanthikumar, J.~G. (2007).  {\em Stochastic Orders}. Springer, New York.
%
 \bibitem[Silverberg and Verspagen, 2007]{silverberg2007size}
Silverberg, G. and Verspagen, B. (2007).
\newblock The size distribution of innovations revisited: An application of extreme value statistics to citation and value measures of patent significance.
\newblock {\em Journal of Econometrics}, {139}(2): 318--339.
%
%
%
%
%
%

\end{thebibliography}
\end{document}